\documentclass[12pt]{article}
\usepackage{amsfonts,color,curves}
\newtheorem{theorem}{Theorem}[section]

\newtheorem{cor}[theorem]{Corollary}
\newtheorem{unnumber}{}

\newtheorem{prepf}[unnumber]{Proof}
\newenvironment{proof}{\prepf\rm}{\endprepf}

\newcommand{\Aut}{\mathop{\mathrm{Aut}}}
\newcommand{\Cay}{\mathop{\mathrm{Cay}}}
\newcommand{\R}{\hbox{\color{red}R}} \newcommand{\B}{\hbox{\color{blue}B}}
\newcommand{\qed}{\qquad$\Box$}

\newcommand{\stsseven}{% must be in a picture with the following declarations
\put(0,0){\circle*{40}}
\put(255,0){\circle*{40}}
\put(510,0){\circle*{40}}
\put(135,225){\circle*{40}}
\put(375,225){\circle*{40}}
\put(255,153){\circle*{40}}
\put(255,425){\circle*{40}}
\put(0,0){\line(1,0){510}}
\put(0,0){\line(5,3){375}}
\put(0,0){\line(3,5){255}}
\put(510,0){\line(-5,3){375}}
\put(510,0){\line(-3,5){255}}
\put(255,0){\line(0,1){425}}
\put(255,144.5){\circle{289}} % this is the largest LaTeX circle
}

\begin{document}

\title{Synchronizing groups and weakly perfect graphs: a survey}
\author{Peter J. Cameron\\
\small{School of Mathematics and Statistics, University of St Andrews,}\\
\small{North Haugh, St Andrews, Fife KY16 9SS, UK}}
\date{}
\maketitle

\begin{abstract}
The concept of synchronization has been transferred from finite automata, via
transformation semigroups, to permutation groups: a permutation group $G$ on
a finite set $\Omega$ is said to be synchronizing if, for any non-permutation
$t$ on $\Omega$, the transformation monoid $\langle G,t\rangle$ contains an
element of rank~$1$. Synchronizing permutation groups are primitive and basic,
and so are affine, diagonal, or almost simple (in the O'Nan--Scott
classification).

In an earlier paper, the author, with Jo\~ao Ara\'ujo and Ben Steinberg,
surveyed what was known at the time about synchronizing permutation groups.
Since then, several new classes of groups have been tested for the
synchronizing property, using the fact that a group is non-synchronizing if
and only if it preserves a nontrivial weakly perfect graph. Cases studied 
include diagonal groups with more than two socle factors, two natural actions
of symmetric groups, and some actions of classical groups; the proofs are
scattered in the literature. My primary purpose is to collect these results
with their proofs. As a second objective, I will explain the background, and
comment on some related topics. (Reading this paper does not require knowledge
of the earlier paper.)



\end{abstract}

\section{Introduction}

In 2017, Jo\~ao Ara\'ujo, Ben Steinberg and I published a long survey paper
about synchronizing permutation groups and related topics. Synchronizing 
groups derived from the theory of synchronization of deterministic automata, 
and the famous \v{C}ern\'y conjecture concerning them; the project was an
attempt to isolate a class of automata for which the conjecture could be
proved. But the topic took on a life of its own, and led to the introduction
of related classes of permutation groups (separating, spreading, and
$\mathbb{Q}I$) lying between primitivity and $2$-transitivity, and connected
with topics such as  representation theory, extremal combinatorics, and finite
geometry as well as automata theory.

The present paper is an update of parts of this survey. But I cast the net
less widely; I am simply concerned with synchronizing groups, where the
property has been decided for several further classes (such as diagonal groups
and actions of symmetric groups on subsets or partitions) since the survey
appeared, and new ideas such as the Hall--Paige conjecture and the
Deza--Erd\H{o}s--Frankl theorem have been put to use.

I have kept the paper self-contained by including definitions and explanations
of the concepts met along the way. However, I expect the reader to have some
basic knowledge about permutation groups. The message of the paper is that
a wide range of results and techniques in discrete mathematics can be applied
to the synchronization problem for permutation groups.

It is dedicated to Mikhail Volkov (from whom I learned much of what I know
about synchronization). I also express my gratitude to Jo\~ao Ara\'ujo, who
first inntroduced me to the subject.

\section{Permutation group properties}

Throughout this paper, $G$ will denote a permutation group on a set $\Omega$
(a subgroup of the symmetric group on $\Omega$). The cardinality of $\Omega$
is the \emph{degree} of $G$, and $G$ is \emph{transitive} if, for any two
points $\alpha,\beta\in\Omega$, there exists $g\in G$ mapping $\alpha$ to
$\beta$.

Many much-studied properties of transitive permutation groups can be defined
in a uniform way: each such property $P$ is associated with a class $C$ of
structures on the set $\Omega$ so that the group $G$ has property $P$ if and
only if there is no non-trivial $G$-invariant $C$-structure. (We say that a
structure is trivial if it is invariant under the symmetric group.) For example,
\begin{itemize}
\item $G$ is primitive if and only if it preserves no non-trivial partition;
\item $G$ is $2$-homogeneous if and only if it preserves no non-trivial
undirected graph;
\item $G$ is $2$-transitive if and only if it preserves no non-trivial directed
graph.
\end{itemize}
For example, if $G$ is not $2$-homogeneous, take an orbit of $G$ on the set
of $2$-subsets of $\Omega$ as the edge set of a graph, which is clearly
$G$-invariant. Conversely, a group preserving a non-trivial graph (that is,
a graph which is not complete or null) cannot map an edge to a non-edge, and
so is not $2$-homogeneous.

The definition of primitivity requires $G$ to be transitive, simply because
there are no non-trivial partitions on $2$ points, and we normally do not want
the trivial group of degree~$2$ to be primitive. For degree greater than $2$,
the assumption of transitivity is unnecessary.

Note that any permutation group property defined in this way is closed under
taking overgroups: if $G\le H$, and $H$ preserves a non-trivial $C$-structure,
then so does $G$.

The study of finite primitive permutation groups goes back to Galois. More
recently, an important tool in their study (part of which was known to Jordan
in the 19th century) is the \emph{O'Nan--Scott theorem}. It is convenient to
separate this theorem into two parts, by means of another definition.

The \emph{Hamming graph} $H(n,q)$ is the graph whose vertex set consists of all
words of length $n$ over an alphabet $Q$ of size $q$, two vertices being 
adjacent if the words agree in all positions except one. It is non-trivial if
$n,q>1$. Now we say that a primitive permutation group is \emph{basic} if it
preserves no non-trivial Hamming graph.

Now the first part of the O'Nan--Scott theorem describes the \emph{socle} (the
product of the minimal normal subgroups) of a group which is primitive but not
basic. The groups here include ordinary and twisted wreath products, and affine
groups whose linear part is an imprimitive linear group. We will see in
Corollary~\ref{c:sync} why these do not concern us.

The second part is as follows:

\begin{theorem}
A basic primitive finite permutation group is affine, diagonal or almost simple.
\label{t:ons2}\qed
\end{theorem}

Here an \emph{affine group} is a group of the form
\[G=\{x\mapsto x^\sigma A+c:c\in F^n,\sigma\in\Aut(F),A\in H\}\]
of permutations of the vector space $F^n$ (where $F$ is a finite field), with
$H$ a group of linear transformations of $F^n$; we see that $G$ is primitive
if and only if $H$ is an irreducible linear group (that is, preserves no
non-zero proper subspace), and basic if and only if $H$ is a primitive linear
group (that is, an irreducible group which preserves no direct sum decomposition
of $F^n$ into more than one subspace). By restriction of scalars, we may assume
that $F$ is a field of prime order; then we don't have to worry about field
automorphisms.

A \emph{diagonal group} is a group whose socle is of the form $T^{d+1}$, where
$T$ is a non-abelian finite simple group, acting by right multiplication on the
right cosets of the diagonal subgroup $\{(g,g,\ldots,g):g\in T\}$; the other
generators of the group are automorphisms of $T$, acting in the same way on all
factors, and elements of the symmetric group $S_{d+1}$ permuting the factors.

An \emph{almost simple group} is a group $G$ satisfying $T\le G\le\Aut(T)$,
where $T$ is a non-abelian finite simple group. (We know from the Classification
of Finite Simple Groups that $\Aut(T)/T$ is very small compared to $T$.)

Except in very small dimensions, affine and diagonal groups can be recognised
by geometric structures that they preserve. For affine groups with $n>1$,
these are the \emph{affine spaces}, whose points are the vectors in $F^n$ and
whose subspaces are the cosets of vector subspaces. For diagonal groups with
$d>1$, they are the \emph{diagonal semilattices} studied in \cite{bcps}.

An important result on almost simple primitive groups, conjectured by the
author, was proved by Tim Burness with a number of different coauthors
\cite{b1,b2,b3,b4}. A
\emph{base} for a permutation group is a sequence of points of the domain
$\Omega$ whose pointwise stabilizer is the identity. Note that if $G$ has a
base of size $b$, then its order is at most $n^b$, where $n$ is the degree
$|\Omega|$.

\begin{theorem}
Let $G$ be an almost simple primitive permutation group. Then one of the
following is true:
\begin{enumerate}
\item $G$ is a symmetric or alternating group, acting on an orbit of subsets
or partitions of its ``natural'' domain;
\item $G$ is a classical group, acting on an orbit of subspaces or pairs of
subspaces of complementary dimension in its ``natural'' module;
\item $G$ has a base of size at most $7$ (with equality only for the Mathieu
group $M_{24}$). \qed
\end{enumerate}
\label{t:base} 
\end{theorem}

Type (c) are ``small'' groups, with polynomially bounded order. Types (a) and
(b) are ``large'' groups. I will say more about types (a) and (b) later in the
paper.

\section{Synchronization}

The topic of synchronization comes from automata theory. The simplest type of
deterministic automaton is a machine which has a set $\Omega$ of internal
states; when a symbol from an alphabet $A$ is input to the machine, its
internal state may change in a definite fashion. In other words, for every
$a\in A$, there is a \emph{transition}, a map $t_a:\Omega\to\Omega$ so that,
if $a$ is input while the machine is in state $\alpha$, it moves to state
$\alpha t_a$. (We write maps on the right). The machine can read the symbols
from a word in $A$ one after the other, and applies the transitions
sequentially. An automaton is \emph{synchronizing} if there is a word $w$ in
the alphabet such that, after reading $w$, the automaton is in a state
depending only on $w$ and not on its initial state; the word $w$ is called a
\emph{reset word}.

Here is an example. The automaton has four states $1,2,3,4$, and the alphabet
has two symbols represented by \R\ and \B.

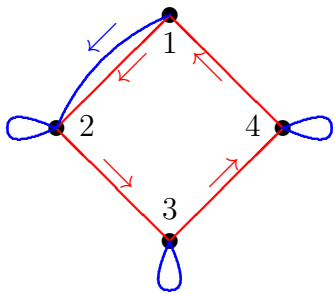
\begin{center}
\setlength{\unitlength}{1mm}
\begin{picture}(50,40)
\thicklines
\multiput(25,10)(0,30){2}{\circle*{2}}
\multiput(10,25)(30,0){2}{\circle*{2}}
\put(24,35){$1$}
\put(13,24){$2$}
\put(24,13){$3$}
\put(35,24){$4$}
\color{red}
\multiput(25,10)(-15,15){2}{\line(1,1){15}}
\multiput(25,10)(15,15){2}{\line(-1,1){15}}
\put(18,32){$\swarrow$}
\put(16,18){$\searrow$}
\put(30,18){$\nearrow$}
\put(28,32){$\nwarrow$}
\color{blue}
\curve(10,25,16,34,25,40)
\put(14,36){$\swarrow$}
\curve(10,25,5,23.5,3.5,25,5,26.5,10,25)
\curve(25,10,23.5,5,25,3.5,26.5,5,25,10)
\curve(40,25,45,23.5,46.5,25,45,26.5,40,25)
\end{picture}
\end{center}
It can be verified that \B\R\R\R\B\R\R\R\B\ is a reset word, mapping every 
state to state $2$; in fact it is the shortest reset word.

A similar construction using a regular $m$-gon in place of a square gives an
$m$-state automaton with a reset word of length $(m-1)^2$. 

Much of the research on synchronizing automata has been driven by the
celebrated \emph{\v{C}ern\'y conjecture}~\cite{cerny}, which states that any
$m$-state synchronizing automaton has a reset word of length at most $(m-1)^2$
(see also \cite{volkov}). At the conference, there were talks by
Antoni Wi\'sniewski and Emanuele Rodaro on topics connected to synchronizing
automata. But I will take a different direction.

We can translate the problem into semigroup theory. The maps generated by the
transitions of the automaton form a transformation monoid $M$ on the set
$\Omega$ of states; the automaton is synchronizing if and only if $M$ contains
an element of \emph{rank}~$1$ (one whose image has cardinality~$1$). So the
\v{C}ern\'y conjecture concerns transformation monoids with a given set of
generators, and we may assume without loss of generality that the generating
set is minimal.

Note that the examples meeting the bound have one of their transitions being a
cyclic permutation. This may have suggested to Ben Steinberg and Jo\~ao
Ara\'ujo that group theory might play a role. The transitions which are
permutations generate a permutation group, which is the group of units of the
transformation monoid. Accordingly, they suggested, take a permutation group
$G$ and look at those transitions $t$ for which $\langle G,t\rangle$ is a
synchronizing monoid. Since a generating set for a transformation monoid must
include a generating set for the group of units, we may assume that the set of
generators which are not permutations is minimal, and begin the investigation
by assuming that there is just one such generator.

Since a permuation group of degree greater than $1$ cannot be synchronizing as
a monoid unless $|\Omega|=1$, we will take over the term and apply it to the
group with a different but related meaning. Thus, we define a permutation
group $G$ to be \emph{synchronizing} if, for any transformation $t$ which is
not a permutation, $\langle G,t\rangle$ contains an element of rank~$1$; in
other words, $G$ \emph{synchronizes} every non-permutation.

It was realised early on that synchronizing groups must be primitive (I will 
give a proof later). The idea that synchronization might be equivalent to
primitivity was refuted by Peter Neumann~\cite{neumann}. He showed that $G$
is non-synchronizing if and only if, for some $k$ with $1<k<|\Omega|$, there
is a $k$-subset $A$ of $\Omega$ and a $k$-part partition $P$ of $\Omega$
such that, for every $g\in G$, the set $Ag$ is a transversal for $A$. (He
called $P$ a \emph{section-regular partition} for $G$.)

To see this, suppose that $\langle G,t\rangle$ is not a synchronizing monoid,
and (without loss of generality) that $t$ has minimal rank in this monoid.
Let $A$ be the image of $t$, and $P$ its kernel (the partition of $\Omega$ 
into inverse images of the points in $A$). If $Ag$ is not a transversal for
$P$, then $|Agt|<|Ag|$, contrary to assumption. Conversely. if $P$ is a 
section-regular partition with section $A$, and $t$ the map taking each part
$B$ of $P$ to the unique point of $A\cap B$, then $\langle G,t\rangle$ is a
non-synchronizing monoid.

In the next section, I will show a definition of ``synchronizating'' in the
form of our earlier definitions of primitive, basic, etc.

\section{The characterization}

Let $\Gamma$ be a simple undirected loopless graph. A \emph{clique} in $A$ is
a set of vertices with every pair adjacent; the \emph{clique number}
$\omega(\Gamma)$ of the graph is the maximum size of a clique. An
\emph{independent set} is a set of vertices containing no edges; the
\emph{independence number} $\alpha(\Gamma)$ is the maximum size of an 
independent set. A (proper) \emph{colouring} of $\Gamma$ is an assignment
of labels or ``colours'' to the vertices so that adjacent vertices have
different colours; equivalently, a partition of the vertex set into independent
subsets. The \emph{chromatic number} $\chi(\Gamma)$ is the least number of
colours required for a colouring of $\Gamma$.

Clearly $\omega(\Gamma)\le\chi(\Gamma)$, since the vertices of a clique must
all have different colours in any proper colouring. A graph $\Gamma$ is said
to be \emph{weakly perfect} if $\omega(\Gamma)=\chi(\Gamma)$. Note that a graph
$\Gamma$ with a clique of size $m$ and a colouring with $m$ colours is weakly
perfect, since
\[m\le\omega(\Gamma)\le\chi(\Gamma)\le m.\]

\begin{theorem}
A permutation group $G$ on $\Omega$ is synchronizing if and only if there is
no non-trivial weakly perfect graph on the vertex set $\Gamma$ which is
preserved by $G$.
\label{t:wp}
\end{theorem}

\begin{proof}
If $\Gamma$ is weakly perfect and $G$-invariant, then the partition of 
$\Omega$ into colour classes in a minimum colouring is section-regular for a 
maximum clique.

Conversely, suppose that $G$ is not synchronizing, and let $P$ be a
section-regular partition with section $A$. Form a graph by joining all pairs
of vertices in $A$ and then including all the images of these edges under
$G$. Then every edge joins vertices in different parts of $P$,  so $P$ defines
a proper colouring, and $\Gamma$ is weakly perfect and $G$-invariant.\qed
\end{proof}

I remark that, if $G$ preserves a weakly perfect graph $\Gamma$, then a
corresponding non-synchronizing monoid with $G$ as group of units is given
concretely by the endomorphism monoid of $\Gamma$.

\medskip

We apply this theorem to show:

\begin{cor}
A synchronizing permutation group is primitive and basic.
\label{c:sync}
\end{cor}

\begin{proof}
If $G$ is imprimitive, then it preserves the disjoint union of  complete
graphs on the parts of a non-trivial $G$-invariant partition; this graph
is clearly weakly perfect.

If $G$ is non-basic, then it preserves a Hamming graph $H(n,q)$, whose
vertices are all $n$-tuples over an alphabet of size $q$. The set of vertices
with fixed values in all but the last coordinate is a clique of size $q$. We
take the alphabet to be the finite field of order $q$, and also take this field
as the set of colours. Give to the vertex $(a_1,a_2,\ldots,a_n)$ the colour
$(a_1+a_2+\cdots+a_n)$. If two vertices are adjacent, they differ in
exactly one coordinate, and so their colours are different. Thus the clique
number and chromatic number are both $q$.\qed
\end{proof}

I remark in passing that Neumann's Example~3.1 in~\cite{neumann} of a primitive
but not synchronizing group is non-basic: it is the automorphism group of
$H(2,3)$.

\section{Separation}

We require also a related concept called separation. Suppose that $G$ is a
transitive permutation group on $\Omega$, and let $A$ and $B$ be subsets of
$\Omega$. Then the average size of $Ag\cap B$ for $g\in G$ is equal to
$|A|\cdot|B|/|\Omega|$. To see this, count triples $(a,b,g)$ such that
$a\in A$, $b\in B$,  $g\in G$, and $ag\in B$. For each point of $A$, there
are $|G|/|\Omega|$ choices of an element of $G$ mapping it to any given point;
so the number of such triples is $|A|\cdot|B|\cdot|G|/|\Omega|$. Dividing by
$|G|$ proves the claim.

In particular, we see that, if $|A|\cdot|B|=|\Omega|$, then the average
number is $1$. So either
\begin{enumerate}
\item for every $g\in G$, we have $|Ag\cap B|=1$; or
\item there exists $g\in G$ with $Ag\cap B=\emptyset$.
\end{enumerate}
If the second alternative holds for all choices of $A$ and $B$ with
$|A|\cdot|B|=|\Omega|$ and $|A|,|B|>1$ (and $G$ is transitive), we say that
$G$ is \emph{separating}.

\begin{theorem}
Let $G$ be a transitive permutation group on $\Omega$.
\begin{enumerate}
\item If $G$ is separating, then it is synchronizing.
\item $G$ is separating if and only if there is no non-trivial graph $\Gamma$
with $\omega(\Gamma)\chi(\Gamma)=|\Omega|$ preserved by $G$.
\end{enumerate}
\end{theorem}

\begin{proof}
(a) If $G$ is non-synchronizing, then there is a $k$-set $A$ and a 
$k$-partition $P$ such that, for all $g\in G$, $Ag$ is a transversal for $P$,
with $1<k<|\Omega|$. If $B$ is a part of $P$, then $|Ag\cap B|=1$ for all
$g\in G$; so $|B|=|\Omega|/k$. Then $A$ and $B$ show that $G$ is non-separating.

(b) Suppose that $A$ and $B$ witness non-separation. Let $\Gamma$ be the graph
whose edge set consists of the $G$-images of the pairs in $A$. Then $A$ is a
clique and $B$ an independent set, and $\Gamma$ is $G$-invariant. Conversely,
if a graph $\Gamma$ with the stated property exists, then a clique and an
independent set, both of maximum size, witness non-separation. \qed
\end{proof}

Both clique number and chromatic number are NP-hard to establish. However,
parametrized complexity gives us concepts which bear out the experience of
programmers that chromatic number is significantly harder than clique number.
(Finding clique size is in the complexity class $W[1]$,
but no $W[1]$ algorithm for chromatic number is known~\cite{bottesch}.)

So, if we are testing computationally whether a permutation group is
synchronizing, it is probably best to test first whether it is separating
(noting that independence number of a graph is just clique number of the
complement); only if this test fails do we need to go on and test 
synchronization. Also, as we will see, synchronization and separation are
equivalent for several classes of groups, including affine and diagonal groups.

Using these ideas and a great deal of ingenuity, Leonard Soicher~\cite{soicher}
has performed a remarkable computation to find all the synchronizing
permutation groups of degree up to $624$. This computation has been very
valuable in shaping our ideas about what to attempt to prove about
synchronizing groups.

Before I describe results for special types of groups, I record a general
result about synchronization and separation. We have seen that separation is
a stronger property. Examples are known to show that it is strictly stronger.
But there is a useful sufficient condition for the equivalence of these two
concepts. This theorem is taken from \cite[Section 3.2]{bccsz}.

Let $G$ be a transitive permutation group on $\Omega$. Suppose that $H$ is a
\emph{regular} subgroup of $G$: this means that $H$ is transitive and the
stabilizer of a point in $H$ is trivial. Then we can identify $\Omega$ with
$H$ by choosing a point $\alpha\in\Omega$ and identifying $h\in H$ with
$\alpha h\in\Omega$. With this identification, $H$ acts on $\Omega$ by right
multiplication.

\begin{theorem}
Let $G$ be a transitive permutation group on $\Omega$. Suppose that $\Omega$
can be identified with a group $H$ such that both right and left multiplication
on $H$ are induced by elements of $G$. If $G$ is synchronizing, then it is
separating.
\label{t:reg}
\end{theorem}

\begin{proof} Suppose that $G$ satisfies the conditions of the theorem.
Identify $\Omega$ with $H$ as explained above. Assume that $G$ is
non-separating, and let $A$ and $B$ be subsets of $H$ witnessing
non-separation. Note that any $G$-image of $A$, with $B$, also witnesses
non-separation. To keep the argument below simple, I use $A$ to mean ``any
$G$-image of $A$''.

\paragraph{Claim 1:} $H$ has an exact factorization by
$A^{-1}=\{a^{-1}:a\in A\}$ and $B$; that is, every element $h\in H$ has a 
unique expression in the form $a^{-1}b$ for $a\in A$, $b\in B$. For, given 
$h\in H$, there is a unique element $b\in Ah\cap B$, and so $h=a^{-1}b$ with
$a\in A$ and $b\in B$; and the expression is unique.

\paragraph{Claim 2:} If $H$ has an exact factorization by $A$ and $B$, then
$G$ is non-synchronizing. We claim that $P=\{Ab:b\in B\}$ is a section-regular
partition with section $B$. The number of sets $Ab$ is equal to $|B|=|H|/|A|$,
and their pairwise disjointness follows from the exact factorization.

\medskip

Now we can conclude the proof. There is a $G$-invariant graph $\Gamma$ on
$H$ in which $A$ is a clique and $B$ an independent set. Since the right action
of $H$ is regular, $\Gamma$ is a Cayley graph for $H$; say $\Gamma=\Cay(H,S)$
for some inverse-closed subset $S$. Since $\Gamma$ admits both the right and
the left actions of $H$, the set $S$ is closed under conjugation. Now $A$
is a clique, so $a_1a_2^{-1}\in S$ for all $a_1,a_2\in A$. But then also
$a_2^{-1}a_1\in S$, and so $A^{-1}$ is also a clique. So by Claim~1, $H$ has
an exact factorization by $(A^{-1})^{-1}=A$ and $B$, and by Claim~2, $G$ is
non-synchronizing. \qed
\end{proof}

\section{Affine groups}

As noted earlier, a primitive affine group $G$ has regular socle $F^n$, where
$F$ is (the additive group of) a finite field of prime order $p$; the point
stabilizer (which is isomorphic to $G/F^n$) is an irreducible linear subgroup
of $\mathrm{GL}(n,p)$.

Since the socle is abelian and regular, left and right translations are the
same; so from Theorem~\ref{t:reg}, we immediately deduce:

\begin{theorem}
A primitive affine group is synchronizing if and only if it is separating.
\label{t:aff} \qed
\end{theorem}

Affine groups may or may not be synchronizing. If $n=1$, then the degree is
prime, and every primitive group is synchronizing. For $n=2$, it is shown
in \cite[Theorem 7.12]{acs} that $G$ is synchronizing if and only if $H$ is
transitive on the $1$-dimensional subspaces of $F^2$. Little is known for
higher dimension. (For an extension of the above result, see
\cite[Theorem 5]{soicher}, but note that this is not a necessary and sufficient
condition.

\section{Diagonal groups}

Recall that a diagonal group has socle $T^{d+1}$, with $T$ a nonabelian finite
simple group, acting on the cosets of the subgroup $\{(t,t,\ldots,t):t\in T\}$;
the subgroup $T^d$ generated by the first $d$ factors is a regular subgroup, so
we can identify $\Omega$ with $T^d$ so that the $i$th factor of the socle acts
by right multiplication on the $i$th coordinate. The $(d+1)$st factor acts
by left multiplication of all coordinates simultaneously by the inverse of the
acting element. The full diagonal group $D(T,d)$ is generated by $T^{d+1}$
together with $\Aut(T)$ acting identically on all coordinates, $S_d$ permuting
the coordinates, and the transposition $(1,d+1)$ which acts as
\[(t_1,t_2,\ldots,t_d)\mapsto(t_1^{-1},t_1^{-1}t_2,\ldots,t_1^{-1}t_d).\]
For more details, see \cite{bcps}.

The two results on diagonal groups from \cite{bccsz} are:

\begin{theorem}
A diagonal group with two factors in the socle is synchronizing if and only if
it is separating.
\label{t:diag1}
\end{theorem}

\begin{theorem}
A diagonal group with more than two factors in the socle is non-synchronizing.
\label{t:diag2}
\end{theorem}

It follows immediately from Theorems~\ref{t:ons2}, \ref{t:aff}, \ref{t:diag1},
and \ref{t:diag2} and Corollary~\ref{c:sync} that a permutation group which
is synchronizing but not separating is almost simple. Examples do exist; we
will see an infinite family in the final section of the paper.

\subsection{Two socle factors}

Specializing the definition, we see that, if $d=1$, the diagonal group acts
on the simple group $T$; one factor of the socle acts by right multiplication,
and the other factor by left multiplication by the inverse. So
Theorem~\ref{t:diag1} follows immediately from Theorem~\ref{t:reg}.

These groups may or may not be synchronizing. Both positive and negative
examples are given in \cite{bglr}, which also illustrates the role played by
group factorization.

\subsection{More than two socle factors}

We now turn to the proof of Theorem~\ref{t:diag2}. This is also taken from
\cite{bccsz}; the proof involves constructing a weakly perfect graph, known
as the diagonal graph, on which the group acts. But the proof requires
another body of theory, outlined below.

The \emph{diagonal graph} $\Gamma_D(T,d)$, for $d\ge2$, can be defined for any
group $T$, finite or infinite (see~\cite{bcps}); we only require the case
where $T$ is a finite simple group. Recall that $\Omega$ is identified with
$T^d$; the diagonal graph is a Cayley graph $\Cay(T^d,S)$, where $S$ is the
union of the $d+1$ direct factors of the socle with the identity removed.
Thus it is invariant under the full diagonal group. The vertex set is $T^d$;
two $d$-tuples are joined if either they disagree in a single coordinate,
or they have the form $(x_1,\ldots,x_d)$ and $(tx_1,\ldots,tx_d)$, for
$x_1,\ldots,x_d,t\in T$ (the last type of join corresponding to the $(d+1)$st
socle factor).

It is easy to see that the clique number of the graph is $|T|$; the cliques
containing the identity correspond to the $d+1$ coordinate subgroups. We must
show that there is a colouring of the graph with colour set also indexed by
$T$.

\paragraph{Case $d$ odd:} In this case, we give the vertex $(x_1,\ldots,x_d)$
the colour\[x_1x_2^{-1}x_3\cdots x_{d-1}^{-1}x_d.\]
Now changing one coordinate clearly changes the colour; and multiplying on the
left by $t$ multiplies the entire expression on the left by $t$ (since the
intermediate factors all cancel). So we do have a proper colouring.

\paragraph{Case $d$ even:} This case is more difficult, and we look first at
the case $d=2$. In this case we can identify the diagonal graph with the
\emph{Latin square graph} associated with the Cayley table of $T$: this is
the graph whose vertices are the cells of the Cayley table, two vertices joined
if the cells lie in the same row, or the same column, or have the same entry.
We make one small adjustment: we re-order the columns by swapping the rows
indexed by $t$ and $t^{-1}$ for all $t\in T$. In other words, the $(x,y)$ entry
is $x^{-1}y$. This ensures that all elements of the socle preserve the graph,
which is thus identified with the diagonal graph.

The Cayley table is a Latin square: that is, each symbol occurs once in each
row and once in each column. A \emph{complete mapping} of a group $G$ is a
bijection $\theta:G\to G$ with the property that the map $\phi$ defined by
$x\phi=x(x\theta)$ is also a bijection. A complete mapping describes a
transversal of a Latin square (the definition ensures that the entries in
row $x$ and column $x\theta$ are all distinct). From the theory of Latin 
squares, it is known that if the Cayley table of a group $G$ has a transversal,
then it has a partition into transversals; that is, there is a (proper)
colouring of the vertices of the Latin square graph with $|G|$ colours.
Explicitly, colour the cell $(x,y)$ of the graph by the colour $x(y^{-1}\phi$).
It is clear that two cells in the same row or column (that is, same value of
$y$ or of $x$) are given different colours, since $\phi$ is a bijection. If
$(x,y)$ and $(u,v)$ contain the same symbol, then $xy^{-1}=uv^{-1}$, so
\[x(y^{-1}\phi)=xy^{-1}(y^{-1}\theta)\ne uv^{-1}(v^{-1}\theta)=u(v^{-1}\phi),\]
since $\theta$ is a bijection.

So we need to know that every non-abelian finite simple group has a complete
mapping.

The general question of complete mappings of groups was investigated by Hall
and Paige~\cite{hp}. They showed that a necessary condition for $G$ to have
a complete mapping is that either $|G|$ has odd order, or the Sylow
$2$-subgroups of $G$ are non-cyclic; they conjectured that these conditions are
also sufficient. They proved their conjecture for solvable groups; the general
proof, which relies on the Classification of Finite Simple Groups, was
completed by Wilcox, Evans and Bray in 2009 \cite{wilcox,evans}. (Bray's part
was not published until the appearance of \cite{bccsz}.)

Since any finite simple group has even order and non-cyclic Sylow $2$-subgroups
(the latter part by Burnside's transfer theorem), we see that
Theorem~\ref{t:diag1} is true for $d=2$.

The extension to all even $d$ is now fairly straightforward. Let $\theta$ be
a complete mapping for $T$, and $\phi$ the corresponding orthomorphism (the
map $x\mapsto x(x\theta)$. Then give the vertex $(x_1,\ldots,x_d)$ the colour
$x_1^{-1}x_2x_3^{-1}\cdots x_{d-1}^{-1}(x_d\phi)$. Once again, changing a
single coordinate changes the colour. A calculation shows that multiplying
each coordinate on the left by $t$ multiplies the colour by 
$(x_d\phi)^{-1}t^{-1}((tx_d)\phi)$. So we need to show that
$(tx_d)\phi\ne t(x_d)\phi$, that is, $tx_d(tx_d)\theta\ne tx_d(x_d)\theta$,
which is true since $\theta$ is a bijection. \qed

\section{Actions of symmetric groups}

We are left with the almost simple groups.

According to Theorem~\ref{t:base}, these can be divided into large and small
groups. The small groups will have more orbits on $2$-sets, so more orbital
graphs; in the absence of general methods to handle them, I will ignore these
groups, and concentrate on the large groups. In this section I discuss the
symmetric and alternating groups. The relevant actions of $S_n$ and $A_n$ are
those on $k$-subsets of $\{1,\ldots,n\}$, or on partitions of $\{1,\ldots,n\}$
into $l$ parts of size $k$ with $kl=n$.

\subsection{On subsets}

The first case to consider is the action of $S_n$ on the set $\Omega$ of
$k$-element subsets of $\{1,\ldots,n\}$. I briefly describe the case $k=2$
before going on to the general case.

The group $S_n$ has two orbits on pairs of $2$-subsets, the intersecting and
disjoint pairs; so we can take the graph $\Gamma$ to be the line graph of the
complete graph $K_n$, where two pairs are joioned if they intersect. For
$n\ge5$, the largest clique consists of all edges through a vertex,
and there do not exist two disjoint such cliques. The largest independent set
has size $\lfloor n/2\rfloor$, and covers all vertices if and only if $n$ is
even. Graph theory textbooks show that, for even $n$, $K_n$ has a
$1$-factorization, a partition of the edge set into $1$-factors or sets of
pairwise disjoint edges covering all vertices. We conclude:

\begin{theorem}
Let $G$ be the symmetric group $S_n$ acting on $2$-subsets. Then the following
are equivalent: $G$ is synchronizing; $G$ is separating; $n$ is odd.
\label{t:2sets} \qed
\end{theorem}

Now for the general case. Since the actions on $k$-subsets
and $(n-k)$-subsets are isomorphic, and the action on $n/2$-subsets is
imprimitive if $n$ is even, we may assume that $k<n/2$. With this assumption,
the intersection of two $k$-subsets has cardinality in the set
$\{0,1,\ldots,k-1\}$, and for each $i$ in this set, the pairs of $k$-sets with
intersection of size $i$ is an orbit of the symmetric group.

Hence, for every subset $L$ of $\{0,1,\ldots,k-1\}$ except $\emptyset$ and
$\{0,1,\ldots,k-1\}$, there is a non-trivial $S_n$-invariant graph in which two
vertices are adjacent if and only if the cardinality of their intersection
belongs to $L$; so there are many cases to check.

The author with Mohammed Aljohani and John Bamberg~\cite{abc} conjectured a
result which cuts through this complexity; and Danila Cherkashin and Jacob
Shubin~\cite{cs}, using a beautiful result of Michel Deza, Paul Erd\H{o}s and
Peter Frankl~\cite{def}, proved the conjecture in a recent paper.

I give the background to the conjecture and outline the proof.

Suppose that $0<t<k<n$. A \emph{Steiner system} $S(t,k,n)$ consists of a set
of $n$ points (which we can take to be $\{1,\ldots,n\}$) together with a
collection $\mathcal{B}$ of $k$-subsets of the point set, with the property
that any $t$-set of points is contained in a unique block in $\mathcal{B}$.

\begin{theorem}
There is a function $F(k)$ with the following property.
Suppose that, for $m\ge F(k)$, the symmetric group $S_n$ acting on $k$-subsets
is non-separating. Then, for some $t$ with $0<t<k$, the non-separation is
witnessed by the graph  $\Gamma_t$ with two subsets joined if their
intersection has cardinality less than $t$; the clique $A$ is the block set
of a Steiner system $S(t,k,n)$, and the independent set consists of all $k$-sets
containing a fixed $t$-set $T$.
\label{t:sn-sep}
\end{theorem}

Before sketching the proof, I make a few comments.

\begin{enumerate}
\item Double counting shows that the number of blocks of a Steiner system
$S(t,k,n)$ is ${n\choose t}/{k\choose t}$, while the number of $k$-sets
containing a fixed $t$-set is ${n-t\choose k-t}$. The product of these numbers
is $n\choose k$, as it should be.
\item Double counting also shows that, for $0\le i\le t-1$, the number of 
blocks of a Steiner system $S(t,k,n)$ containing a fixed set of $i$ blocks is
${n-i\choose t-i}/{k-i\choose t-i}$. So a necessary condition for the existence
of such a system is that all these numbers should be integers. (These are the
\emph{divisibilitly conditions}.)

A remarkable result of Peter Keevash~\cite{keevash} settled asymptotically a
problem which had been open since Wesley Woolhouse posed it in the \emph{Lady's
and Gentleman's Diary} in 1844~\cite{woolhouse}. Keevash showed that, for
sufficiently large $n$ (in terms of $k$ and $t$), the divisibilitly conditions
are also sufficient for the existence of a Steiner system $S(t,k,n)$. This
means that, for sufficiently large $n$, the question whether $S_n$ acting on
$k$-sets is separating or not can be decided by simple arithmetic (checking the
divisibility conditions for all $t$).
\item What about synchronization? It is easy to see that the set of $k$-sets
cannot be partitioned into independent sets of the type described in the
theorem. So, for $n$ as in the theorem, the action of $S_n$ on $k$-sets is
non-synchronizing if and only if, for some $t$, the set of $k$-sets can be
partitioned into block sets of Steiner systems $S(t,k,n)$ for some fixed $t$.
Now for $t=1$, the only divisibility condition is that $k$ divides $n$. In
this case, a Steiner system is a set of $k$-sets partitioning the point set,
and the theorem of Baranyai~\cite{baranyai} shows that the partition exists.
So $S_n$ on $k$-sets is non-synchronizing whenever $k$ divides $n$. For larger
values of $t$, the question of when the set of $k$-sets can be partitioned into
Steiner systems is unsolved except in very few cases. For example, when $t=2$
and $k=3$, the divisibility conditions show that $n$ is congruent to $1$ or
$3$ (mod~$6$); Kirkman~\cite{kirkman} constructed Steiner systems $S(2,3,n)$
for all $n$ satisfying these conditions, and Lu~\cite{lu} and
Teirlinck~\cite{teirlinck} showed that the partition exists for all such $n$
except $n=7$.
\end{enumerate}

In the case $n=7$, the non-existence of the partition was already known to
Cayley~\cite{cayley}): in fact there are no more than two disjoint $S(2,3,7)$
systems. But a different construction shows that $S_7$ on $3$-sets is
non-synchronizing. Take $A$ to be the set of lines of the Fano plane,
the unique (up to isomorphism) $S(2,3,7)$ (Figure~\ref{f:fano}).

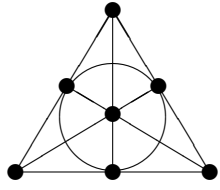
\begin{figure}[htbp]
\begin{center}
\setlength{\unitlength}{0.0505mm}
\begin{picture}(510,425)
\stsseven
\end{picture}
\end{center}
\caption{\label{f:fano}The Fano plane}
\end{figure}

For each line $L$ in $A$, let $S(L)$ consist of $L$ together with the four
$3$-sets disjoint from $L$. Then $|S(L)|=5$; and $S(L_1)\cap S(L_2)=\emptyset$
for $L_1\ne L_2$, since no $3$-set is disjoint from two lines. So the seven
sets $S(L)$ form a partition. Moreover, $Ag$ is
a transversal to this partition, for any $g\in S_7$; so the partition is
section-regular with section $A$, and $S_7$ is non-synchronizing.

\medskip

Next I state the Deza--Erd\H{o}s--Frankl theorem, and the Erd\H{o}s--Ko--Rado
theorem which we also need for the proof.

\begin{theorem}[Deza--Erd\H{o}s--Frankl] There is a function $f$ with the
following property: Given $k$, let $L$ be a subset of $\{0,1,\ldots,k-1\}$ with
$|L|=t$, say
\[L=\{l_1<l_2<\cdots<l_t\}.\]
Define a graph $\Gamma_L$ whose vertices are the $k$-subsets
of $\{1,\ldots,n\}$, with adjacency defined by the cardinality of the
intersection belonging to $L$. Let $A$ be a clique in this graph.
If $n\ge f(k)$ and $|A|>k^2(2n)^{t-1}$, then
\begin{enumerate}
\item $(l_2-l_1)\mid(l_3-l_2)\mid\cdots\mid(l_t-l_{t-1}\mid(k-l_t)$;
\item $|A|\le\prod_{i=1}^t(n-l_i)/(k-l_i)$. \qed
\end{enumerate}
\label{t:def}
\end{theorem}

\begin{theorem}[Erd\H{o}s--Ko--Rado]
There is a function $g$ such that, if a family $\mathcal{F}$ of $k$-subsets of
an $n$-set with $n\ge g(k)$ have the property that any two intersect in at
least $t$ points, then $|\mathcal{F}|\le{n-t\choose k-t}$, with equality only
if $\mathcal{F}$ consists of all the $k$-sets containing a fixed $t$-set.
\label{t:ekr} \qed
\end{theorem}

Now we outline the proof of Theorem~\ref{t:sn-sep} by Cherkashin and Shubin.
If $S_n$ on $k$-sets is not separating, then there is a subset $L$ of
$\{0,\ldots,k-1\}$ with $A$ a clique and $B$ an independent set in
$\Gamma_L$ and $|A|\cdot|B|={n\choose k}$. Also $B$ is a clique in the graph
$\Gamma_{L'}$, where $L'=\{0,1,\ldots,k-1\}\setminus L$, so $|L'|=k-t$.
By the last statement of Theorem~\ref{t:def}, if $|L|=t$, then $|A|=O(n^t)$ and
$|B|=O(n^{k-t})$. Now we cannot have $|A|=O(n^{t-1})$ or $|B|=O(n^{k-t-1})$,
since $|A|\cdot|B|=\Omega(n^k)$; so the conclusions of the theorem hold for 
both $A$ and $B$, for large enough $n$.

The element $k-1$ must belong to either $L$ or $L'$; we may assume without loss
that it is in $L'$, so $k-l'_{k-t}=1$. Then the divisibility conditions show
that $l'_i-l'_{i-1}=1$ for all $i$; that is, $L'=\{t,t+1,\ldots,k-1\}$, and
$L=\{0,1,\ldots,t-1\}$. So $A$ consists of sets meeting in at most $t-1$ points,
Thus, any $t$-set lies in at most one set in $A$.
A simple counting argument shows that $|A|\le{n\choose t}/{k\choose t}$.

Also any two sets in $B$ meet in at least $t$ points, so Theorem~\ref{t:ekr}
shows that $|B|\le{n-t\choose k-t}$. Thus $|A|\cdot|B|\le{n\choose k}$. Since
we have equality here, we must have equality in the bounds for $|A|$ and $|B|$,
and the theorem follows. \qed

\medskip

The theorems of Baranyai, Lu, and Teirlinck, and the special construction for
$n=7$, show:
\begin{theorem}
Let $G$ be the symmetric group $S_n$ acting on $3$-sets. Then the following are
equivalent: $G$ is synchronizing; $G$ is separating; $n\equiv2$, $4$ or $5$
$(\mathrm{mod}~6)$. \qed
\end{theorem}

\subsection{On partitions}

The situation for partitions is more clear-cut. The following result is from
\cite{cdr}. It is interesting that, in most results of this kind, we construct
a graph and prove that it is weakly perfect: typically the clique number is
easy to find and the chromatic number more difficult. But here it is the other
way around!

\begin{theorem}
Let $G$ be the permutation group induced by the symmetric group $S_{kl}$ on
the set of all partitions into $l$ subsets of size $k$, where $k>1$ and $l>2$.
Then $G$ is not synchronizing.
\end{theorem}

\begin{proof}
We say that a partition into $l$ sets of size $k$ has type $k^l$.
We construct a graph on the vertex set $\Omega$ of all partitions of type $k^l$
as follows: join two partitions if they have no common part.
We claim that this graph has chromatic number at most ${kl-1}\choose{k-1}$. To
see this, take a point $x$, and colour the vertices by the $(k-1)$-subsets of
$\{1,\ldots,n\}\setminus\{x\}$ by giving a vertex $P$ the colour $A$ if the
part of $P$ containing $x$ is $A\cup\{x\}$. Two vertices with the same colour
have a part in common and so are nonadjacent, so we have a proper colouring.

For the clique number, we use the theorem of Baranyai~\cite{baranyai},
according to which the set of all $k$-subsets of $\{1,\ldots,n\}$ can be
partitioned into classes each of which forms a partition of $\{1,\ldots,n\}$,
where $n=kl$. Each class has $l$ parts; so the number of classes is
\[{kl\choose k}\bigg/l={kl-1\choose k-1}.\]
The classes are pairwise disjoint, so the partitions are pairwise adjacent.

So the clique number is at least ${kl-1\choose k-1}$, and the theorem is proved.
\qed
\end{proof}

For $l=2$, the graph is complete, and the argument fails. The authors of
\cite{cdr} point out that, for $l=2$ and $l=3$, the group is $2$-transitive
and hence synchronizing. (For $S_4$ on the three $2^2$ partitions is isomorphic
to $S_3$, and $S_6$ on the ten $3^2$ partitions to
$\mathrm{P}\Sigma\mathrm{L}(2,9)$.) For $l=4$ and $l=6$ it is
non-synchronizing; but for $n=5$ a computation shows that it is synchronizing.

We can see that $S_8$ on $4^2$ partitions is non-synchronizing as follows. There
is a bijection between $3$-subsets of $\{1,\ldots,7\}$ and $4^2$ partitions
of $\{1,\ldots,8\}$ as follows: $L$ maps to the partition
$\{L\cup\{8\},\{1,\ldots,7\}\setminus L\}$. Now the construction using the
Fano plane in the preceding subsection maps to a section-regular partition
for $S_8$ on $4^2$ partitions.

\section{Classical groups}

The other large groups are classical groups acting on an orbit of subspaces or
pairs of subspaces of complementary dimension. Much less is known in this
case, so the account will be brief.

The first family of classical almost simple groups are the groups
$\mathrm{P}\Gamma\mathrm{L}(n,q)$, generated by invertible linear maps on
$F^n$, where $F$ is the finite field of order $q$, modulo scalars, together
with field automorphisms. So first we need to consider the action of this
group on the projective space whose points are the $1$-dimensional subspaces
of $F^n$, and its action on the set of $k$-dimensional subspaces.

This group is a $q$-analogue of the symmetric group acting on subsets.
However, there is no $q$-analogue of the Deza--Erd\H{o}s--Frankl theorem
(as far as I know), and the theory of ``Steiner systems over finite fields''
(collections of $k$-subspaces over $F^n$ with the property that any
$t$-subspace is contained in exactly one space in the collection) is in its
infancy: the first example (with $(t,k,n)=(2,3,13)$ over the field of $2$
elements) was found by Michael Braun, Tuvi Etzion, Patric {\O}stergard,
Alexander Vardy and Alfred Wassermann~\cite{beovw} in 2016.

There is one case in which more is known, the case $t=1$. A \emph{spread} is
a family $S$ of $k$-dimensional subspaces of $F^n$, where $F$ is a field of 
order~$q$, such that every non-zero vector lies in just one member of $S$.
If a spread exists, then $q^k-1$ (the number of nonzero vectors in a $k$-space)
divides $q^n-1$, so $k\mid n$. Conversely, if $k\mid n$, take a vector space
of dimension $n/k$ over a field of order $q^k$: its $1$-dimensional subspaces
partition the non-zero vectors. Now restricting scalars to the subfield $F$ of
order $q$ gives a spread of $k$-spaces in $F^n$. For $n/k>2$, all spreads are
of this form; for $n/k=2$, there are many others known, and they are closely
connected with the theory of finite projective planes (specifically
\emph{translation planes}).

A \emph{parallelism}, or \emph{packing}, of $F^n$ by $k$-spaces is a
partition of the set of all $k$-spaces into spreads.  Parallelisms have been
much less studied. Existence has been shown in two cases, both with $k=2$:
the case where the field has order~$2$ (by Baker~\cite{baker}), and the case
where $n$ is a power of~$2$ (by Beutelspacher~\cite{beutelspacher}, building on
the case $n=4$ which was settled earlier by Denniston~\cite{denniston}).

To summarize the conclusions:

\begin{theorem}
Let $G$ be the group $\mathrm{P}\Gamma\mathrm{L}(n,q)$, acting on the set of
$k$-subspaces of $F^n$, where $F$ has order $q$.
\begin{enumerate}
\item If $k$ divides $n$, or if $(n,k)=(13,3)$, then $G$ is non-separating.
\item If $k=2$ and $n$ is odd, then $G$ is separating.
\item If $k=2$ and either (i) $n$ is even and $q=2$, or (ii) $n$ is a power
of $2$, then $G$ is non-synchronizing. \qed
\end{enumerate}
\end{theorem}

The argument for part (b) relies on the fact that the group acts transitively
on pairs of $2$-spaces with each possible dimension ($0$ or $1$) of
intersection, and the only clique in the disjointness graph which is large
enough is a spread. The transitivity holds also for the subgroup
$\mathrm{PSL}(n,q)$ induced by transformations with determinant~$1$, so
separation holds also for this subgroup.

So the $q$-analogue of Theorem~\ref{t:2sets} holds for $q=2$.
It may be that the $q$-analogues of Theorem~\ref{t:sn-sep} and Keevash's
Theorem hold in full generality, but we are some way from proving that.

\medskip

Things are even more complicated for the other classical groups. These preserve
non-degenerate symplectic (alternating bilinear), Hermitian or quadratic forms
on a finite vector space. The geometries preserved by these groups are the
\emph{polar spaces} over finite fields, axiomatised and classified in a major
work by Jacques Tits~\cite{tits} in 1974.

A polar space consists of a set of points with a collection of subsets called
subspaces, such that a subspace together with the subspaces it contains form
a projective space. (There are other axioms which we will not need). The
classical projective spaces are named after the groups (symplectic, unitary
or orthogonal) containing them; the points are the $1$-dimensional subspaces
on which the form defining the group vanishes. The account below ignores a
certain complication for orthogonal groups in characteristic $2$.

The group acts with rank~$3$ on the set of points (this means it has just two
orbits on pairs of points, namely orthogonal and non-orthogonal with respect
to the form (in the symplectic and unitary cases) or the bilinear form obtained
by polarizing the quadratic form (in the orthogonal case). Thus there is a
single complementary pair of $G$-invariant graphs; we will use the orthogonality
graph. The maximal cliques in this graph are the maximal subspaces in the polar
space. Thus, an upper bound for the size of an independent set is $R/S$, where
$R$ is the number of points in the space and $S$ the number of points in a 
maximal clique. A set meeting this bound is called an \emph{ovoid}, and is
characterised by the property that it meets each maximal subspace in one point.
Also, a \emph{spread} is a partition of the points into maximal subspaces.
So we have:

\begin{theorem}
Let $G$ be a classical group, acting on the points of its polar space.
\begin{enumerate}
\item $G$ is separating if and only if the polar space has an ovoid.
\item $G$ is synchronizing if and only if the polar space has either an ovoid
and a spread, or a partition into ovoids. \qed
\end{enumerate}
\end{theorem}

There has been a lot of work in finite geometry on ovoids and spreads in
classical polar spaces, though a complete answer is not yet known:
see \cite{bglr2}. In particular, the polar spaces
associated with non-degenerate quadratic forms on $5$-dimensional vector spaces
over odd prime fields have ovoids, but neither spreads nor partitions into
ovoids. This gives us the example of an infinite family of permutation groups
which are synchronizing but not separating, promised earlier.

Bamberg \emph{et al.}~\cite{bglr2} also begin the study of classical groups
acting on the set of subspaces of fixed dimension greater than~$1$.

\section{Other groups}

We have no general methods for deciding whether ``small'' primitive groups
(of polynomially-bounded order) are synchronizing; this is probably the
biggest open problem.

It is instructive to go through the list of non-synchronizing groups of small
degree \cite[Table 3]{soicher}, seeking a geometric explanation for the
existence of a weakly perfect $G$-invariant graph. Here are two examples.
\begin{itemize}
\item $M_{12}.2$, with degree $144$. This group has an orbital graph which
is isomorphic to the Hamming graph $H(2,12)$. (The action is on the Cartesian
product of two sets of size~$12$ affording the two inequivalent $5$-transitive
actions of the Mathieu group $M_{12}$, interchanged by the outer automorphism.)
\item $\mathrm{P}\Sigma\mathrm{U}(3,5)$,  with degree $175$. This group
acts on the line graph of the Hoffman--Singleton graph
$\mathrm{HoSi}$~\cite{hs}. Its clique number is the valency of
$\mathrm{HoSi}$, namely $7$; and its chromatic number is the chromatic index
or edge-chromatic number of $\mathrm{HoSi}$, which was shown to also be $7$ by
S.~M. Cioab\u{a}, K. Guo, and W.~H. Haemers~\cite{cgh}. (Soicher attributes this
fact also to Gordon Royle.) Indeed, the authors of \cite{cgh} conjecture that
all strongly regular graphs on an even number of vertices apart from the
Petersen graph have chromatic index equal to their valency, which would 
imply that the automorphism groups of their line graphs are non-synchronizing.
\end{itemize}


\begin{thebibliography}{99}

\bibitem{abc}
Mohammed Aljohani, John Bamberg and Peter J. Cameron,
Synchronization and separation in the Johnson scheme,
\textit{Portugaliae Mathematica} \textbf{74} (2018), 213--232; doi:
\texttt{10.4171/PM/2003}.

\bibitem{acs}
Jo\~ao Ara\'ujo, Peter J. Cameron and Benjamin Steinberg,
Between primitive and 2-transitive: Synchronization and its friends,
\textit{Europ. Math. Soc. Surveys} \textbf{4} (2017), 101--184; doi:
\texttt{10.4171/EMSS/4-2-1}.

\bibitem{bcps}
R. A. Bailey, Peter J. Cameron, Cheryl E. Praeger and Csaba Schneider,
The geometry of diagonal groups,
\textit{Trans. Amer. Math. Soc.} \textbf{375} (2022), 5259--5311; doi:
\texttt{10.1090/tran/8507}.

\bibitem{baker}
R. D. Baker,
Partitioning the planes of $\mathrm{AG}_{2m}(2)$ into $2$-designs,
\textit{Discrete Math.} \textbf{15} (1976), 205--211; doi:
\texttt{10.1016/0012-365X(76)90025-X}

\bibitem{bglr}
John Bamberg, Michael Giudici, Jesse Lansdown and Gordon F. Royle,
Synchronising primitive groups of diagonal type exist,
\textit{Bull. London Math. Soc.} \textbf{54} (2022), 1131--1144; doi:
\texttt{10.1112/blms.12619}

\bibitem{bglr2}
John Bamberg, Michael Giudici, Jesse Lansdown and Gordon F. Royle,
Tactical decompositions in finite projective spaces and non-spreading
classical group actions,
\textit{Designs, Codes, Cryptography} \textbf{93} (2025), 1127--1141; doi:
\texttt{10.1007/s10623-024-01490-y}

\bibitem{baranyai}
Z. Baranyai,
On the factorization of the complete uniform hypergraph,
\textit{Colloq. Math. Soc. Janos Bolyai} \textbf{10} (1975), 91--108.

\bibitem{beutelspacher}
Albrecht Beutelspacher,
On parallelisms in ﬁinite projective spaces,
\textit{Geom. Dedicata.} \textbf{3} (1974), 35--40; doi:
\texttt{10.1007/BF00181359}

\bibitem{bottesch}
Ralph C. Bottesch,
On $W[1]$ hardness as evidence for intractability,
in: Potapov, Igor (ed.) et al., 43rd international symposium on mathematical
foundations of computer science (MFCS 2018),
Schloss Dagstuhl -- Leibniz Zentrum für Informatik. LIPIcs -- Leibniz Int.
Proc. Inform. \textbf{117} (2018), Article 73, 15pp.; doi:
\texttt{10.4230/LIPIcs.MFCS.2018.73}.

\bibitem{beovw}
Michael Braun, Tuvi Etzion, Patric {\O}stergard, Alexander Vardy and Alfred
Wassermann,
Existence of $q$-analogues of Steiner systems,
\textit{Forum Math. Pi} \textbf{4}, Paper No. e7, 14pp. (2016); doi:
\texttt{10.1017/fmp.2016.5}.

\bibitem{bccsz}
John N. Bray, Qi Cai, Peter J. Cameron, Pablo Spiga and Hua Zhang,
The Hall--Paige conjecture, and synchronization for affine and diagonal groups,
\textit{J. Algebra} \textbf{545} (2020), 27--42; doi:
\texttt{10.1016/j.jalgebra.2019.02.025}

\bibitem{b1}
Timothy C. Burness, 
On base sizes for actions of finite classical groups,
\textit{J. London Math. Soc.} (2) \textbf{75} (2007), 545--562; doi:
\texttt{10.1112/jlms/jdm033}

\bibitem{b2}
Timothy C. Burness, Martin W. Liebeck and Aner Shalev,
Base sizes for simple groups and a conjecture of Cameron,
\textit{Proc. London Math. Soc.} (3) \textbf{98} (2009), 116--162; doi:
\texttt{10.1112/plms/pdn024}

\bibitem{b3}
Timothy C. Burness, E. A. O'Brien and Robert A. Wilson,
Base sizes for sporadic simple groups,
\textit{Israel J. Math.} \textbf{177} (2010), 307--333; doi:
\texttt{10.1007/s11856-010-0048-3}

\bibitem{b4}
Timothy C. Burness, Robert M. Guralnick and Jan Saxl,
On base sizes for symmetric groups,
Bull. Lond. Math. Soc. 43 (2011), 386--391; doi:
\texttt{10.1112/blms/bdq123}

\bibitem{cdr}
Peter J. Cameron, Coen del Valle, and Colva M. Roney-Dougal,
Regular bipartite multigraphs have many (but not too many) symmetries,
\textit{Discrete Analysis}, October 2025; doi:
\texttt{10.19086/da.144894}

\bibitem{cayley}
A. Cayley,
On the triadic arrangements of seven and fifteen things,
\textit{London, Edinburgh and Dublin Philos. Mag. and J. Sci.}
\textbf{3} (1850) no. 37, 50--53.

\bibitem{cerny}
J. \v{C}ern\'y,
Pozn\'amka homog\'ennym eksperimentom s konen\'ymi automatami
[A remark on homogeneous experiments with finite automata],
\textit{Mat.-Fyz. \v{C}asopis Slovensk. Akad. Vied.} \textbf{14} (1964),
49--75.

\bibitem{cs}
Danila Cherkashin and Jacob Shubin,
Short proofs of three combinatorial results in the Johnson scheme,
arXiv \texttt{2605.30092} (2026).

\bibitem{cgh}
S. M. Cioab\u{a}, K. Guo and W.~H. Haemers,
The chromatic index of strongly regular graphs,
\textit{Ars Math.  Contemp.} \textbf{20} (2021), 187--194; doi:
\texttt{10.26493/1855-3974.2435.ec9}

\bibitem{denniston}
Ralph H. F. Denniston,
Some packings of projective spaces,
\textit{Atti Accad. Naz. Lincei, Cl. Scienze Fisiche,
Matematiche e Naturali, Rendiconti}, (8), \textbf{52} (1972), 36--40.

\bibitem{def}
M. Deza, Paul Erd\H{o}s and P. Frankl,
Intersection properties of systems of finite sets,
\textit{Proc. Lond. Math. Soc.} (3) \textbf{36} (1978), 369--384; doi:
\texttt{10.1112/plms/s3-36.2.369}.

\bibitem{ekr}
P\'aul Erd\H{o}s, Chao Ko and Richard Rado,
Intersection theorems for systems of finite sets,
\textit{Quart. J. Math. Oxford} II, \textbf{12} (1961), 313--320; doi:
\texttt{10.1093/qmath/12.1.313}.

\bibitem{evans}
A. B. Evans,
The admissibility of sporadic simple groups,
\textit{J. Algebra} \textbf{321} (2009), 105--116; doi:
\texttt{10.1016/j.jalgebra.2008.09.028}.

\bibitem{hp}
M. Hall Jr. and L. J. Paige,
Complete mappings of finite groups,
\textit{Pacific J. Math.} \textbf{5} (1955), 541--549; doi:
\texttt{10.2140/pjm.1955.5.541}.

\bibitem{hs}
A. J. Hoffman and R. R. Singleton,
On Moore graphs with diameters $2$ and $3$,
\textit{IBM J. Res. Dev.} \textbf{4} (1960), 497--504; doi:
\texttt{10.1147/rd.45.0497}. 

\bibitem{keevash}
Peter Keevash,
The existence of designs,
arXiv \texttt{1401.3665}.

\bibitem{kirkman}
Thomas P. Kirkman,
On a problem in combinations,
\textit{Cambridge and Dublin Mathematical Journal} (1847), 191--204.

\bibitem{lu}
Jia-Xi Lu,
On large sets of disjoint Steiner triple systems I-VI,
\textit{J. Combinatorial Theory} (A) \textbf{34} (1983), 140--192.

\bibitem{neumann}
Peter M. Neumann,
Primitive groups and their section-regular partitions,
\textit{Michigan Math. J.} \textbf{58} (2009), 309--322; doi:
\texttt{10.1307/mmj/1242071695}.

\bibitem{soicher}
Leonard H. Soicher,
The non-synchronizing primitive groups of degree up to $624$,
\textit{J. Algebra} \textbf{703} (2026) 422--434; doi:
\texttt{10.1016/j.jalgebra.2025.08.003}

\bibitem{teirlinck}
Luc Teirlinck,
A completion of Lu's determination of the spectrum for large sets of disjoint
Steiner triple systems,
\textit{J. Combionatorial Theory} (A) \textbf{57} (1991), 302--305; doi: 
\texttt{10.1016/0097-3165(91)90053-J}

\bibitem{tits}
Jacques Tits,
\textit{Buildings of spherical type and finite BN-pairs},
Lecture Notes in Mathematics \textbf{386}, Springer Verlag,
Berlin--Heidelberg--New York 1974; doi:
\texttt{10.1007/978-3-540-38349-9}. 

\bibitem{volkov}
Mikhail V. Volkov,
Synchronizing automata and the \v{C}ern\'y conjecture,
in \textit{Mathematical Foundations of Computer Science 2006}, 11-27,
Lecture Notes in Comp. Sci. \textbf{5196}, Springer, Berlin, 2008; doi:
\texttt{10.1007/978-3-540-88282-4\_4}.

\bibitem{wilcox}
S. Wilcox,
Reduction of the Hall--Paige conjecture to sporadic simple groups,
\textit{J. Algebra} \textbf{321} (2009), 1407--1428; doi:
\texttt{10.1016/j.jalgebra.2008.11.033}.

\bibitem{woolhouse}
W. S. B. Woolhouse,
Prize Question 1733,
\textit{Lady's and Gentleman's Diary}, 1844.

\end{thebibliography}
\end{document}